\documentclass[10pt,english]{amsart}

\usepackage[text={6in,8.5in},centering]{geometry}
\usepackage{dsfont,mathptmx,newtxtext,newtxmath}
\usepackage[all]{xy}
\usepackage[dvipsnames]{xcolor}   
\usepackage{xparse}
\usepackage{xr-hyper}
\usepackage[linktocpage=true,colorlinks=true,hyperindex,citecolor=Brown,linkcolor=blue]{hyperref}  

\makeatletter
\def\msection{\@startsection{section} %name
	{1} % level
	{0pt} % indent
	{-1ex plus -.1ex minus -0.9ex} % beforeskip
	{-.9ex plus -.2ex} % afterskip
	{\bfseries} % style 
}
\def\msubsection{\@startsection{subsection} %name
	{2} % level
	{0pt} % indent
	{-1ex plus -.1ex minus -0.2ex} % beforeskip
	{-.9ex plus -.2ex} % afterskip
	{\normalfont} % style
} 
\makeatother

\newcommand{\inv}{^{\raisebox{.2ex}{$\scriptscriptstyle-1$}}}  

\newtheorem{theorem}{Theorem}[section]
\newtheorem{proposition}[theorem]{Proposition}
\newtheorem{lemma}[theorem]{Lemma}

\theoremstyle{definition}

\numberwithin{equation}{section}
\newtheorem{question}[theorem]{Problem}

\numberwithin{equation}{section}

\begin{document}

\title{On the Extensions of the Cohen Structure Theorem}

\author{Amartya Goswami}

\address{[1]  Department of Mathematics and Applied Mathematics, University of Johannesburg, P.O. Box 524, Auckland Park 2006, South Africa. [2]  National Institute for Theoretical and Computational Sciences (NITheCS), South Africa.}

\email{agoswami@uj.ac.za}


%Local rings and semilocal rings: None of the above, but in this section

\begin{abstract}
The purpose of this note is to pose a question that, when answered, would directly imply the Cohen Structure Theorem. We provide a solution to this question for a specific class of local rings (not necessarily complete). We also explore how this question connects to the Gelfand-Mazur Theorem and, more broadly, to a fundamental theorem in Gelfand theory. Finally, we provide a categorical reformulation of the question.
\end{abstract}

\maketitle 
 
\section{Introduction}

Let $A$ be a local ring with maximal ideal $\mathfrak{m}$. $A$ is said to be an \emph{equicharacteristic}
local ring if $A$ has the same characteristic as its residue field $A/\mathfrak{m}$. A \emph{field of representatives} for $A$
is a subfield $\kappa$ in $A$ that maps onto $A/\mathfrak{m}$ under the canonical projection. Since
$\kappa$ is a field, the restriction of this mapping to $\kappa$ gives an isomorphism between $\kappa$ and $A/\mathfrak{m}$.
The
Cohen structure theorem  says that

\begin{theorem}[\cite{C46}]
An equicharacteristic complete local ring $A$ admits a field of representatives.
\end{theorem}

It is clear that if a ring $A$ contains a field $\kappa$, then it is equicharacteristic. We aim to formulate a problem for a much broader class of rings such that Cohen structure theorem becomes a special case of it. From now on, all our rings are assumed to be commutative and to possess an identity element.

\begin{question}
\label{mq}
Characterize a ring $A$ that contains a field $\kappa$ and satisfy that $A/\mathfrak{m}$ is isomorphic to $ \kappa$, for some maximal ideal $\mathfrak{m}$ in $A$. 
\end{question}

Problem \ref{mq} indeed makes sense. Here are some examples. 

\begin{enumerate}
\item[$\bullet$] Every field \( \kappa \) trivially has the property. 
	
\item[$\bullet$] Every polynomial ring \( \kappa[x_1, \dots, x_n] \) over a field $\kappa$ with maximal ideal \( \langle x_1, \dots, x_n\rangle\).	
	
\item[$\bullet$]
The localization \( \kappa[x]_{\langle x\rangle} \) of \( \kappa[x] \) at the maximal ideal $\langle x\rangle$.
	
\item[$\bullet$]  The ring of formal power series \( \kappa\llbracket x_1, \ldots, x_n\rrbracket \) over a field \( \kappa \) with maximal ideal $\langle x_1,\ldots, x_n\rangle$.
	
\item[$\bullet$] For a field $\kappa$, the ring \( \kappa \times \kappa \) with maximal ideals \( \kappa \times \{0\} \) and \( \{0\} \times \kappa \).
\end{enumerate}

To set up the ground, let us settle with the  notation and terminologies.  We denote the set of maximal ideals in a ring $A$ by $\mathrm{Max}(A),$ whereas the notation $\mathfrak{m}^*$  denotes the set of all non-zero elements in a maximal ideal $\mathfrak{m}.$ By $U(A)$, we denote the set of invertible elements in $A$. If $\kappa$ is a field, then the set of all invertible elements in $\kappa$  is denoted by  $\kappa^*$.   When $\kappa\subseteq A$, then we sometimes represent it as the canonical identity map $\iota\colon \kappa\to A,$ and we  write  $\iota(\kappa)$. Such a ring $A$ can be seen as a $\kappa$-algebra. Whenever we say `$A$ is a $\kappa$-algebra', we want to mean there exists a copy of the field $\kappa$ in $A$.  
The set of all $\kappa$-algebra homomorphisms from $A$ to $\kappa$ is denoted by $\mathrm{Hom}_{\mathrm{Alg}_{\kappa}}(A, \kappa).$   

\section{Some results}

We begin by considering a restricted version of Problem \ref{mq}: Which $\kappa$-algebra $A$ with $\kappa$ as the largest field in $A$ has the property that $A/\mathfrak{m}$ is isomorphic to $\kappa$, for some $\mathfrak{m}\in \mathrm{Max}(A)?$ 
The next theorem gives such a class of $\kappa$-algebras, but first, two lemmas.

\begin{lemma}\label{dsum}
Let $(A, \mathfrak{m})$ be a local $\kappa$-algebra. If \,$\kappa$ is the largest field in $A,$ then   \[U(A)=\{ u+m\mid u\in   \kappa^*\;\text{and}\; m\in \mathfrak{m} \}.\] 
\end{lemma}

\begin{proof}
Let \( a \in U(A) \). Then  the image \(\overline{a}\) of $a$ in the field \( A/\mathfrak{m} \) is invertible. Corresponding to this \(\overline{a}\),
there exist \( u \in \kappa^* \)  (since  $\kappa$  is the largest field in $A$) and \( m \in \mathfrak{m} \) such that \( a = u + m \). This shows that
\[
U(A) \subseteq \{ u + m \mid u \in \kappa^*, \, m \in \mathfrak{m} \}.
\]
Now, consider an element \( u + m \), where \( u \in \kappa^* \) and \( m \in \mathfrak{m} \). We claim that \( u + m \in U(A)\). Indeed, $u+m \in \mathfrak{m}$ would imply $(u+m)-m=u \in \mathfrak{m},$ a contradiction. 
Since $u+m\notin \mathfrak{m}$ and since $\mathfrak{m}$ is the unique maximal ideal in $A$, we infer that $u+m$ is invertible. 
\end{proof}

\begin{lemma}\label{mt1}
If $A$ is an $\kappa$-algebra and $\mathfrak{m}\in \mathrm{Max}(A),$ then $A/\mathfrak{m}$ is isomorphic to $\kappa$ if and only if $A= \kappa\oplus \mathfrak{m}.$  Moreover, if $\kappa'$ is a field with $A= \kappa'\oplus \mathfrak{m},$ then $\kappa'$ is isomorphic to $\kappa.$  
\end{lemma}

\begin{proof}	First we show that if  $\pi\colon  A\to A/\mathfrak{m} $ is an isomorphism, then  $A= \kappa \oplus \mathfrak{m}.$ Since, in particular, $\pi$ is surjective, we have $\pi(\kappa)=A/\mathfrak{m},$ and  hence \[\pi\inv \pi  (\kappa)= \kappa\oplus \mathrm{ker}(\pi)=\kappa\oplus \mathfrak{m}= \pi\inv(A/\mathfrak{m})=A.\] To prove the converse, let $A= \kappa\oplus \mathfrak{m}.$ This implies that \[A/\mathfrak{m}=\pi(A)=\pi(\kappa\oplus \mathfrak{m})=\pi(\kappa)\oplus \pi(\mathfrak{m})=\pi(\kappa),\] that is, $\pi$ is injective. Since $\pi$ is also surjective, $\pi$ is an isomorphism. The proof of the second part of the theorem is obvious. 
\end{proof}

\begin{theorem}\label{qfld}
If $(A, \mathfrak{m})$ is a local $\kappa$-algebra with $\kappa$ the largest field in $A,$ then  $A/\mathfrak{m}$ is isomorphic to $\kappa.$  
\end{theorem}

\begin{proof}
Although the proof follows from Lemma \ref{mt1} and Lemma \ref{dsum},
however, here is an alternative and independent proof.

Define a  map
$\varphi\colon A/\mathfrak{m} \to \kappa$ by
\( \varphi(a + \mathfrak{m}) = a \) for \( a \in A \setminus \mathfrak{m} \). This map is well-defined because if \( a + \mathfrak{m} = b + \mathfrak{m} \), then \( a - b \in \mathfrak{m} \). Since \( a \) and \( b \) are invertible elements in \(\kappa\), it follows that  \(a-b=0\).
The map \( \varphi \) is injective. Indeed: if \( \varphi(a + \mathfrak{m}) = 0 \) in \(\kappa\), then \( a \in \mathfrak{m}\), meaning that \( a + \mathfrak{m} = \mathfrak{m}\). Thus, the kernel of \( \varphi \) consists only of the zero element \( \mathfrak{m} \).
Now, for any element \( u \in \kappa^* \), which is a unit in \( A \), we can find \( u + \mathfrak{m} \in A/\mathfrak{m} \) such that \( \varphi(u + \mathfrak{m}) = u \). Thus,  \( \varphi \) is surjective. 
\end{proof}

Is the unique maximality of $\kappa$ in Theorem \ref{qfld} preserved under the converse formulation? We have a positive answer with respect to the maximality, however, we are not certain on the uniqueness part. Notice that in the next result we do not need $A$ to be local.

\begin{proposition}\label{mt2}
If $\kappa$ is a field in $A$ and $\kappa$ is isomorphic to $A/\mathfrak{m}$ for some maximal ideal $\mathfrak{m}$ in $A$,  then $\kappa$ is a maximal field in $A$.
\end{proposition} 

\begin{proof}
If possible, suppose there is field $\kappa'$ such that  $\kappa\subsetneq \kappa'\subsetneq A.$ Then $\pi(\kappa)\subseteq \pi(\kappa')\subseteq \pi(A)=A/\mathfrak{m}.$ But $\pi(\kappa)=A/\mathfrak{m},$ and this implies that $\pi(\kappa)=\pi(\kappa'),$ which subsequently gives \[\pi\inv \pi (\kappa)=\kappa\oplus \mathfrak{m}=A=\pi\inv\pi(\kappa')=\kappa'\oplus \mathfrak{m}.\] Hence, $\kappa$ is isomorphic to $\kappa'.$
\end{proof}

When $\kappa$ is the field $\mathds{C}$ of complex numbers, the following example brought to my attention by Partha Pratim Ghosh, and this example, once again, gives evidence in favour of Problem \ref{mq}.

Consider a Tychonoff topological space \( X \) and the ring \( C(X,\mathds{C}) \) of all continuous complex-valued functions on \( X \). Obviously, \( C(X,\mathds{C}) \) is a commutative unitary ring as well as a vector space over \( \mathds{C} \) with \[ z(fg) = (zf)g = f(zg), \] for any \( z \in \mathds{C}\) and \( f, \) \(g \in C(X,\mathds{C}) \). The ring  has a copy of \( \mathds{C} \) embedded in it as a \( \mathds{C} \)-subalgebra.
Let \( \mathfrak{m} \) be a maximal ideal of \( C(X,\mathds{C}) \); obviously \( C(X,\mathds{C})/\mathfrak{m}\) is a field. The assignment \( p \mapsto p/\mathfrak{m}\) is an embedding of \( \mathds{C} \) in \( C(X,\mathds{C})/\mathfrak{m}\), and hence the quotient contains a copy of \( \mathds{C} \). If \( \mathfrak{m} \) be fixed, that is, \[ \bigcap_{f \in \mathfrak{m}} \{ x \in X \mid f(x) = 0 \} \] is non-empty, then from the Tychonoff property it is a singleton set, and in such cases the quotient map is an isomorphism. 

The choice of the field $\mathds{C}$ for our original problem leads to another question on generalization of a well-known result in ``Gelfand theory'', namely

\begin{theorem}[\cite{Dal03}]
If $A$ is a unital, commutative Banach algebra over $\mathds{C}$, then there is a bijection between   $\mathrm{Max}(A)$
and the nonzero multiplicative linear functionals on $A$.
\end{theorem}

It is easy to see that as a very special case of the above, we have following celebrated result.

\begin{theorem}[Gelfand-Mazur Theorem] If $A$ is a unital, commutative Banach algebra over $\mathds{C}$ in which every nonzero element is invertible, then $A$ is isometrically isomorphic to $\mathds{C}$.
\end{theorem}
 Let us ask a general question on $\mathds{C}$-algebras of which the last two results will be special cases.

\begin{question}\label{qfa}
For which $\mathds{C}$-algebra $A,$ we have a bijection between the sets $\mathrm{Max}(A)$ and $\mathrm{Hom}_{\mathrm{Alg}_{\mathds{C}}}(A, \mathds{C})?$ 
\end{question}

Notice that Problem \ref{qfa} is a special case of our original Problem \ref{mq}, for $\kappa=\mathds{C}$ and enforcing the isomorphism $A/\mathfrak{m} \approx\kappa$, for every $\mathfrak{m} \in \mathrm{Max}(A)$. To see that this is a valid question, let us consider the $\mathds{C}$-algebra $\mathds{C}[x],$ where \(\mathrm{Max}\left(\mathds{C}[x]\right)=\{ \langle x-a \rangle\mid a\in \mathds{C}\},\) and for each $\langle x-a \rangle\in \mathrm{Max}(\mathds{C}[x]),$ we have $\mathds{C}[x]/\langle x-a \rangle$ isomorphic to $\mathds{C}.$  

\section{A categorical formulation}

Finally, we shall see a ``structural formulation'' of our Problem \ref{mq}. The following categorical version  has been formulated for me by George Janelidze.

Let $R$ be a ring and $I$ be an ideal of $R$. The ring $R$ is said to be \emph{complete with respect to} $I$ if 
\[R=\lim\limits_{n}\left(R/I^n\right).\]
In particular, $R$ is said to be \emph{complete} if $R$ is a local ring complete with respect to its unique maximal ideal.
Let us consider the following two extreme cases:

\begin{enumerate}
\item[$\bullet$] If $I$ is idempotent (that is, $I^2=I$), then $\lim_{n}\left(R/I^n\right)=R/I$, and so $R$ is complete with respect to $I$ if and only if $I=0$.
	
\item[$\bullet$] If $I$ is nilpotent (that is, $I^n=0$, for some $n\in \mathds{N}$), then $\lim_{n}\left(R/I^n\right)=R/0=R,$ and so $R$ is always complete with respect to $I$.
\end{enumerate}

We are interested in the situations where the canonical homomorphism $R\to R/I$ is a split epimorphism of rings (with $1$). In this case, we have the diagram:
\[
\xymatrix@R=1cm@C=1cm{
	I\ar@{=}[d]\ar[r]^{i} & R\ar@<.4ex>[r]^p\ar@<.4ex>[d]^{\psi} & R/I\ar@{=}[d]\ar@<.4ex>[l]^s\\
	I\ar[r]^(.3){\iota_1} & I\rtimes (R/I)\ar@<.4ex>[r]^(.6){\pi_2}\ar@<.4ex>[u]^{\phi} & R/I\ar@<.4ex>[l]^(.4){\iota_2}
}
\]
in the category fo non-unital commutative rings, in which:

\begin{enumerate}
	\item[$\bullet$] $p$ is defined by $p(r)=r+I$;
	
	\item[$\bullet$] $s$ is any unital ring homomorphism with $ps=1_{R/I};$
	
	\item[$\bullet$] the underlying abelian group of $I\rtimes  (R/I)$ is $I\times (R/I)$;
	
	\item[$\bullet$] the multiplication of $R\rtimes(R/I)$ is defined by \[(x,u)(y,v)=(xy+s(v)x+s(u)y, uv);\]
	
	\item[$\bullet$] the remaining homomorphisms are defined by 
	\begin{itemize}
		\item[]$\bullet$ $i(x)=x$;
		
		\item[]$\bullet$ $\phi(x,u)=x+s(u)$;
		
		\item[]$\bullet$ $\psi(r)=(r-sp(r), p(r))$;
		
		\item[]$\bullet$ $\iota_1(x)=(x,0)$;
		
		\item[]$\bullet$ $\iota_2(x,u)=u$;
		
		\item[]$\bullet$ $\iota_2(u)= (0,u)$.
	\end{itemize}
	\item[$\bullet$] $I\rtimes (R/I)$ is a unital ring with $1=(0,1)$, and $\phi$, $\psi$, $\pi_2$, and $\iota_2$ preserve $1$.
	
	\item[$\bullet$] it is easy to check that 
	\[\phi\iota_1=i, p\phi=\pi_2, \psi s=\iota_2,\pi_2\iota_2=1_{R/I},\]
	and that $\phi$ and $\psi$ are isomorphisms inverse to each other;
	
	\item[$\bullet$] $s$ makes $R$ an $R/I$-algebra, $I$ and $R/I$-subalgebra of it, and (therefore) makes $i$ a homomorphism of $R/$-algebras. Of course, $R$ is unital, whereas $I$ does not have to be so. 
\end{enumerate}
This determines an equivalence $\mathds{S}\text{plit}\mathds{E}\text{xt}\sim \mathds{A}\text{lg}$ between the triples $(R,I,s)$ above and pairs $(K,A)$, where $K$ is a unital commutative ring and $A$ (not necessarily) commutative $K$-algebra. For a pair $(K,A)$, the corresponding triple is
\[(A \rtimes K, \{(a,k)\in A \rtimes K\mid k=0\},\quad k\mapsto (0,k)),\]
where in $A \rtimes K$, we have
\[ (a,k)(b,l)=(ab+la+kb,kl).\]
With the above set up, let us now reformulate our main questions.

\begin{question}
Is there any example, where
$K$ is a field and
$A\rtimes K$ is local but not complete?
\end{question}

\begin{question}
Can we characterize $A\rtimes K$ in such a way that Cohen structure theorem becomes a consequence of it?
\end{question}

Once again, in two extreme cases, we have the following scenarios:
\begin{itemize}
\item[$\bullet$] If $A$ is unital with $1\neq 0$, then $A\rtimes K$ is not local. Indeed:
\[(-1,1)(a,k)=(0,1)\Rightarrow (k=1\wedge -a+k+a=0) \Rightarrow 1=k=0\] and so $(-1,1)$ cannot be invertible.
	
\item[$\bullet$] If $A^2=0$, then $A$ is complete. Indeed, whenever $k\neq 0$, we have
\[ (a,k)(-k^{-2}a, k^{-1})=(k^{-1}a+k(-k^{-2}a), kk^{-1})=(k^{-1}a-k^{-1}a, 1)=(0,1),\]
and so $(a,k)$ is invertible for every $a\in A$.
\end{itemize}

%\section*{Acknowledgement}   

\end{document}